\documentclass{birkmult}
\usepackage{amsmath}
\usepackage{amsthm}
\usepackage{amssymb}
\usepackage{eucal}
\usepackage{cite}
\newtheorem{pro}{Proposition}[section]

\newtheorem{thm}[pro]{Theorem}
\newtheorem{corl}[pro]{Corollary}
\theoremstyle{remark}

\theoremstyle{definition}

\numberwithin{equation}{section}

\begin{document}
%-------------------------------------------------------------------------
% editorial commands: to be inserted by the editorial office
%
%\firstpage{1}
%\volume{228}
%\Copyrightyear{2004}
%\DOI{003-0001}
%
%
%\seriesextra{Just an add-on}
%\seriesextraline{This is the Concrete Title of this Book\br H.E. R and S.T.C. W, Eds.}
%
% for journals:
%
%\firstpage{1}
%\issuenumber{1}
%\Volumeandyear{1 (2004)}
%\Copyrightyear{2004}
%\DOI{003-xxxx-y}
%\Signet
%\commby{inhouse}
%\submitted{March 14, 2003}
%\received{March 16, 2000}
%\revised{June 1, 2000}
%\accepted{July 22, 2000}
%
%
%
%---------------------------------------------------------------------------
%Insert here the title, affiliations and abstract:
%

%\bibliographystyle{short}

\title[Non- (quantum) differentiable $C^1$-functions]{Non- (quantum)
  differentiable $C^1$-functions in the spaces with trivial Boyd
  indices}
%----------Author 1
\author[D.Potapov]{Denis Potapov}

% \address{%
% School of Informatics and Engineering, \br
% Faculty of Science and Engineering, \br
% Flinders Univ. of SA, Bedford Park, 5042, \br
% Adelaide, SA, Australia.}

\email{pota0002@infoeng.flinders.edu.au}

% \thanks{This work was completed with the support of our
% \TeX-pert.}
%----------Author 2
\author[F.Sukochev]{Fyodor Sukochev}
\address{%
School of Informatics and Engineering, \br
Faculty of Science and Engineering, \br
Flinders Univ. of SA, Bedford Park, 5042, \br
Adelaide, SA, Australia.}
\email{sukochev@infoeng.flinders.edu.au}
%----------classification, keywords, date


\date{October 31, 2006}
%----------additions
%\dedicatory{To my boss}
%%% ----------------------------------------------------------------------
\maketitle
%%% ----------------------------------------------------------------------
%\tableofcontents

{ %%% HERE I PUT GROUPING TO MAKE MY MACROS LOCAL TO THIS GROUP, SO
  %%% THAT THEM WON'T INTERVENE WITH ANY OTHERS OUTSIDE.  DO NOT TOUCH.

%%% MY PERSONAL MACROS %%%

\renewcommand{\labelenumi}{{\rm \theenumi}}

\def\toRomans{
  \renewcommand{\theenumi}{(\roman{enumi})}}

\def\toAlpha{
  \renewcommand{\theenumi}{(\alph{enumi})}}

\toRomans

\let\ucal\mathcal
\let\ds\displaystyle
\let\ss\scriptstyle
\def\Rl{{\mathbb{R}}}
\def\Bd{{\mathcal{B}}}
\def\Z{{\mathbb{Z}}}
\def\N{{\mathbb{N}}}
\def\Rank{{\mathcal{F}}}
\def\HFF{{\mathcal{R}}}
\def\id{{\mathbf{1}}}
\def\cC{{\mathfrak{S}}}
\def\Cx{{\mathbb{C}}}
\def\M#1{{\mathbb{M}_{#1}(\Cx)}}
\def\Mf#1{{M_f(#1)}}
\def\H{{\mathcal H}}
\def\de{=}
\def\eq#1{\mathop{\smash{=}}\limits^{\hbox{\small #1}}}
\def\geq#1{\mathop{\smash{\ge}}\limits^{\hbox{\small #1}}}
\def\leq#1{\mathop{\smash{\le}}\limits^{\hbox{\small #1}}}

%%% END OF MY PERSONAL MACROS %%%

%%% ----------------------------------------------------------------------
\begin{abstract}
  If~$E$ is a separable symmetric sequence space with trivial Boyd
  indices and~$\cC^E$ is the corresponding ideal of compact operators,
  then there exists a ~$C^1$-function~$f_E$, a self-adjoint
  element~$W\in \cC^E$ and a densely defined closed symmetric
  derivation~$\delta$ on~$\cC^E$ such that~$W \in Dom\ \delta$,
  but~$f_E(W) \notin Dom\ \delta$.
\end{abstract}
%%% ----------------------------------------------------------------------

\section{Introduction.}

This paper studies properties of infinitesimal generator~$\delta^\cC$ of
a strongly continuous group~$\alpha = \{\alpha_t\}_{t\in\Rl}$ in Banach
algebras~$\cC \subseteq \Bd(\H)$, given by~$\alpha_t(y) = e^{itX} y
e^{-itX}$, $y \in \cC$, where~$X$ is an unbounded self-adjoint operator
in the Hilbert space~$\H$.  The generator~$\delta^\cC$ is a densely
defined closed symmetric derivation on~$\cC$ and we are concerned with
the question when its domain~$Dom\ \delta^\cC$ satisfies the following
condition $$ x = x^* \in Dom\ \delta^\cC \Rightarrow f(x) \in Dom\
\delta^\cC, $$ for every~$C^1$-function~$f: \Rl \rightarrow \Cx$?  In
the recent paper~\cite{AdPS2005}, it is shown that there
are~$C^*$-algebras~$\cC$ and operators~$X$ for which the implication
above fails (see also~\cite{McI1978}).  In this paper, we consider the
case when the Banach algebra~$\cC$ is a symmetrically normed ideal of
compact operators on~$\H$. (see e.g.\ \cite{GohbergKrein} and
Section~\ref{sec:symSpaces} below).  It is immediately clear that for
every self-adjoint operator~$X$, the group~$\alpha$ acts isometrically
on such an ideal~$\cC$ and, in fact, is a~$C_0$-group on~$\cC$, provided
that~$\cC$ is separable (see e.g.~\cite{PS-RFlow}).  It is an
interesting problem to determine the class of ideals~$\cC$ in
which~$Dom\ \delta^\cC$ is closed with respect to the~$C^1$-functional
calculus.  Note, that the class of such ideals is non-empty.  For
example it contains the Hilbert-Schmidt ideal.  The proof of the latter
claim may be found in~\cite{PoSu}.  On the other hand, it is unclear
whether this class contains the Schatten-von Neumann ideals~$\cC^p$
when~$1< p < \infty$, $p \ne 2$.  In this paper, we however show that
the class of all symmetrically normed ideals~$\cC$ whose Boyd indices
are trivial fail the implication above.  For various geometric
characterizations of the latter class we refer to~\cite{Ar1978} (see
also Section~\ref{sec:symSpaces} below).  Our methods are built upon and
extend those of~\cite{AdPS2005, McI1978}.  Our results also contribute
to the study of commutator bounded operator-functions initiated
in~\cite{KiShI, KiShII, KiShIII}.

\section{Schur multipliers.}

Let $\M n$ be the $C^*$-algebra of all $n\times n$ complex matrices,
let $B\in \M n$ be a diagonal matrix $diag\{\lambda_1,\lambda_2,
\ldots, \lambda_n\}$.  The Schur multiplier $\Mf B$ associated with
the diagonal matrix $B$ and the function $f:\Rl\mapsto\Cx$ is defined
as follows. For every matrix $X=\{\xi_{jk}\}_{j,k=1}^n\in\M n$, the
matrix $\Mf B(X)\in\M n$ has $(j,k)$ entry given by $$ [\Mf B(X)]_{jk}
= \psi_f(\lambda_j,\lambda_k)\xi_{jk},\ \ 1\le j,k\le n, $$ where
$$
\psi_f(\lambda,\mu) = \begin{cases} \ds {f(\lambda)-f(\mu)\over
    \lambda - \mu},& \text{$\lambda\ne\mu$,}\cr 0,&
  \text{$\lambda=\mu$.}\cr
\end{cases} $$ Alternatively, if $\{P_j\}_{j=1}^n$ is the collection
of one-dimensional spectral projections of the matrix $B$ then
$B=\sum_{j=1}^n \lambda_jP_j$, and
\begin{equation} \label{schurprojections} 
\Mf BX = \sum_{1\le j,k\le n}
\psi_f(\lambda_j,\lambda_k)P_jXP_k.   
\end{equation}
For every matrix $X\in\M n$ the following equation outlines the
interplay between the Schur multiplier~$\Mf B$ and the commutator~$[B,
X] = BX - XB$
\begin{equation} \label{schurcomm} 
\Mf
B([B,X]) = [f(B),X].  \end{equation}
Indeed, 
\begin{align*}
  \Mf B([B,X]) & = \sum_{1\le j, k\le n} \psi_f(\lambda_j,\lambda_k)
  P_j\Bigl[\sum_{s=1}^n\lambda_sP_s,X\Bigr]P_k \cr & = \sum_{1\le j,
    k\le n} \psi_f(\lambda_j,\lambda_k) (\lambda_j-\lambda_k)P_jXP_k \cr
   & = \sum_{1\le j, k\le n} (f(\lambda_j) - f(\lambda_k))P_jXP_k \cr & =
  \sum_{1\le j, k\le n} P_j\Bigl[\sum_{s=1}^nf(\lambda_s)P_s, X\Bigr]P_k
  = [f(B), X].\cr
\end{align*}

\section{Symmetric spaces with trivial Boyd indices.}
\label{sec:symSpaces}

Let~$E=E(0, \infty)$ be a symmetric Banach function space, i.e.\ $E=E(0,
\infty)$ is a rearrangement invariant Banach function space on~$(0,
\infty)$ (see~\cite{LT-II}) with the additional property that~$f, g\in
E$ and~$g \prec\prec f$ imply that~$\|g\|_E \le \|f\|_E$.  Here~$g\prec
\prec f$ denotes submajorization in the sense of Hardy, Littlewood and
Polya, i.e.\ $$ \int_0^t g^*(s)\, ds \le \int_0^t f^*(s) \, ds,\ \ t >
0, $$ where~$f^*$ (respectively,~$g^*$) stands for the decreasing
rearrangement of the function~$f$ (respectively,~$g$).

Let us consider the group of dilations~$\{\sigma_\tau\}_{\tau > 0}$
defined on the space~$S = S(0, \infty)$ of all Lebesgue measurable
functions on~$(0, \infty)$.  The operator~$\sigma_\tau$, $\tau > 0$ is
given by $$ (\sigma_\tau f)(t) = f(\tau^{-1}\, t),\ \ t > \infty. $$
If~$E$ is a symmetric Banach function space, then the lower
(respectively, upper) Boyd index~$\alpha_E$ (respectively,~$\beta_E$) of
the space~$E$ is defined by $$ \alpha_E := \lim_{\tau \rightarrow +0}
\frac{\log \|\sigma_\tau\|_{E \mapsto E}}{\log \tau}\ \
\left(\text{respectively,}\ \beta_E := \lim_{\tau \rightarrow +\infty} \frac
{\log \|\sigma_\tau\|_{E \mapsto E}}{\log \tau}\right). $$

We say that the space~$E$ has the trivial lower (resp.\ upper) Boyd
index when~$\alpha_E = 0$ (respectively, $\beta_E = 1$).  It is known
that, if~$\alpha_E = 0$ (respectively,~$\beta_E = 1$), then the
space~$E$ is not an interpolation space in the pair~$(L_1, L_p)$ for
every~$p < \infty$ (respectively,~$(L_q, L_\infty)$ for every~$1< q$),
\cite[Section~2.b]{LT-II}.

\begin{pro}  
  (\cite[Proposition 2.b.7]{LT-II}) If $E$ be a symmetric sequence space
  and $\alpha_E=0$ (respectively, $\beta_E=1$), then for every
  $\varepsilon>0$ and every $n\in\N$ there exist $n$ disjointly
  supported vectors $\{x_j\}_{j=1}^n$ in $E$, having the same
  distribution, such that for every scalars $\{a_j\}_{j=1}^n$ the
  following holds \begin{align} \label{trivboydestimates} \max_{1\le
      j\le n} |a_j| & \le \Bigl\|\sum_{j=1}^n a_jx_j \Bigr\|_E \le
    (1+\varepsilon) \max_{1\le j\le n} |a_j|\cr
    \biggl(\hbox{respectively, }(1-\varepsilon) \sum_{j=1}^n |a_j| &\le
    \Bigl\|\sum_{j=1}^n a_jx_j\Bigr\|_E \le \sum_{j=1}^n |a_j|
    \biggr).\cr \end{align} \end{pro}

\noindent If $E$ is separable then $x_j$ can be chosen finitely
supported. 

$\cC^E$~denotes the corresponding symmetric ideal of compact operators
on the Hilbert space~$\ell_2 = \ell_2(\N)$, i.e.\ the space of all
compact operators~$x$ such that~$s(x) \in E$, where~$s(x)$ is the step
function such that $$ s(x)(t) = s_k,\ \ k < t \le k+1,\ \ k \ge 0 $$
and~$\{s_k\}_{k \ge 0}$ the sequence of singular numbers (counted with
multiplicities) of the operator~$x$ (see e.g.~\cite{GohbergKrein}).  The
norm in the space~$\cC^E$ is given by~$\|x\|_{\cC^E} := \|s(x)\|_{E}$.
In particular, if~$E = L_p$, then the ideal~$\cC^p = \cC^{L_p}$, $1\le p
< \infty$ stands for the Schatten-von Neumann ideals of compact
operators and~$\cC^\infty$ stands for the ideal of all compact operators
equipped with the operator norm, see~\cite{GohbergKrein}.

Let~$\ell_2^n$ be the subspace in~$\ell_2$ spanned by the first~$n$
standard unit vector basis.  If an element~$B \in \cC^E$ is such
that~$B = B|_{\ell_2^n}$, then we identify~$B$ with its matrix
from~$\M n$.

\begin{pro} 
  \label{phipsi} Let $E$ be a separable symmetric function space and
  $\alpha_E=0$ (respectively, $\beta_E=1$).  For every scalar
  $\varepsilon>0$ and every positive integer $n\in\N$ there exist linear
  operators $\Phi_n$ and $\Psi_n$ such that
  \begin{enumerate}
  \item\label{aaai} $\Phi_n,\Psi_n:\M n\mapsto \M {k_n}$, where
    $\{k_n\}_{n \ge 1}$ is a sequence of positive integers;
  \item\label{aaaii} the operators $\Phi_n$, $\Psi_n$ map diagonal
    (respectively, self-adjoint) matrices to diagonal (respectively,
    self-adjoint) matrices; 
  \item\label{aaaiii} if $\Mf B$, $\Mf{\Phi_n(B)}$ are the Schur
    multiplier associated with the diagonal matrices $B\in\M n$,
    $\Phi_n(B) \in \M {k_n}$ and the function $f$, then $\Psi_n(\Mf
    BX) = \Mf{\Phi_n(B)}\Psi_n(X)$ for every matrix $X\in\M n$;
  \item\label{aaaiv} $\|X\|_{\cC^\infty}\le \|\Psi_n(X)\|_{\cC^E} \le
    (1+\varepsilon) \|X\|_{\cC^\infty}$ (respectively,
    $(1-\varepsilon)\|X\|_{\cC^1}\le \|\Psi_n(X)\|_{\cC^E} \le
    \|X\|_{\cC^1}$) for every matrix $X\in\M n$. \end{enumerate}
\end{pro}

\begin{proof} Let $n$ be a fixed positive integer and $\varepsilon>0$
  be a fixed positive scalar. Let $\{x_j\}_{j=1}^n$ be a sequence of
  finitely and disjointly supported vectors, having the same
  distribution such that~\eqref{trivboydestimates} holds.  Let~$X_0$
  be the matrix given by $X_0=diag\{x^*_1(k)\}_{k\ge 1}$, i.e. $X_0$
  is the finite diagonal matrix in $\cC^E$ that corresponds to the
  decreasing rearrangement $x^*_1$ in $E$.  Let~$I$ be the
  identity matrix of the same size as~$X_0$.  We define the linear
  operators~$\Phi_n$ and~$\Psi_n$ by $$ \Phi_n(X) := X \otimes
  I,\ \ \text{and}\ \ \Psi_n(X) := X \otimes X_0,\ \ X\in\M n. $$ The
  claims~\ref{aaai}, \ref{aaaii}, now, follow immediately from the
  definition of $\Phi_n$ and $\Psi_n$ and the claim~\ref{aaaiii}
  follows from~\eqref{schurprojections}.
  
  Let us prove~\ref{aaaiv}. For every matrix $X \in\M n$ there exist
  unitary matrices $U,V$ such that $$ UXV = diag\{s_1, s_2, \ldots,
  s_n\}. $$ Now, it follows from elementary properties of tensors, that
  \begin{align*}
    \Phi_n(U)\Psi_n(X)\Phi_n(V) & = (U\otimes I)(X\otimes X_0)(V\otimes
    I) \cr & = (UXV)\otimes X_0 = \Psi_n(UXV)
    =diag\{s_jX_0\}_{j=1}^n,\cr
  \end{align*} 
  and so
  \begin{align*}
    \|\Psi_n(X)\|_{\cC^E} &= \|\Phi_n(U)\Psi_n(X)\Phi_n(V)\|_{\cC^E} \cr
    & = \|diag\{s_jX_0\}_{j=1}^n\|_{\cC^E} = \Bigl\|\sum_{j=1}^n
    s_jx_j\Bigr\|_E.\cr
  \end{align*}
  Now, the claim in~\ref{aaaiv} for $\alpha_E=0$ (respectively,
  $\beta_E=1$) follows from combining the equality above with the first
  estimate in~\eqref{trivboydestimates} (respectively, the second
  estimate in~\eqref{trivboydestimates}).\end{proof}

\noindent The operators $\Phi_n$, $\Psi_n$ are very similar to those,
constructed in the proof of~\cite[Theorem~4.1]{Ar1978}.

\section{Commutator estimates.}

From now on let $h:\Rl\mapsto\Rl$ be a function with the following
properties \toAlpha
\begin{enumerate}
\item $h(t)\in C^1(\Rl\setminus\{0\})$; \label{FL}
\item\label{hevenness}$h(t)=h(-t)$ when $t\ne 0$, $h(0)\ge 0$; 
\item $h(\cdot)$ is
increasing function on~$(0,\infty)$; \item\label{hinf}%
$h(\pm\infty)=+\infty$; \item $0\le h'(t) / h(t) \le 1$ when
$t\in(0,\infty)$.\label{hmainprop} \label{LL}
\end{enumerate} \toRomans

\begin{pro} \label{ffuncPro} Let $h(t)$ be a function that satisfies the
  conditions~\ref{FL}--\ref{LL} above.  If $f$ is a function defined as
  follows
\begin{equation} 
  \label{ffunc}
  f(t) = \begin{cases} |t|(h(\log |t|))^{-1}, & 
    \text{if $|t|<1$, $t\ne
    0$,} \cr 0, & \text{if
    $t=0$.}\cr \end{cases} \end{equation}
  then $f(t)\in C^1(-1,1)$ and $f'(t)\ge
  0$ for every $t\in(0,1)$. \end{pro}

\begin{proof} 
  The function given in~\eqref{ffunc} is even so it is sufficient to
  consider only the case $t\ge 0$. It follows from the definition of the
  function $f$ that for every $t\in(0,1)$ function $f$ is continuously
  differentiable. To calculate the derivative at zero, we use the
  definition $$
  f'(0)=\lim_{t\rightarrow 0} {f(t)-f(0)\over t-0} =
  \lim_{t\rightarrow 0} (h(\log t))^{-1} \eq{\ref{hinf}} 0.  $$
  In order
  to verify that $f'(t)\rightarrow 0$ when $t\rightarrow +0$, we note
  first that $$
  f'(t) = (h(\log t))^{-1} \biggl(1- {h'(\log t)\over
    h(\log t)}\biggr),\ \ 0<t<1. $$
  Since $h(t)\ge0$ for every
  $t\in\Rl$, together with the property~\ref{hmainprop}, it now follows
  that for every $t\in(0,1)$ $$
  0\le f'(t) \le 2(h(\log
  t))^{-1}\rightarrow 0, \hbox{ as } t\rightarrow+0. $$\end{proof}

\noindent Let matrices $D,V\in\M m$ and $A,B\in\M{2m}$ be defined as
follows \begin{align} \label{DV} D & = diag\{e^{-1},e^{-2},\ldots,
  e^{-m}\}, \cr V & = \{v_{jk}\}_{j,k=1}^m, \cr v_{jk} & = \begin{cases}
    (k-j)^{-1} (e^{-j}+e^{-k})^{-1}, & \text{if $j\ne k$,} \cr 0, &
    \text{ if $j=k$.} \cr \end{cases}, \cr \end{align} and
\begin{equation} \label{AB} A=\left[
    \begin{matrix} 0 & V \cr -V & 0 \cr
    \end{matrix} \right],\ \ \ 
  B=\left[\begin{matrix} D & 0 \cr 0 & -D \cr \end{matrix} \right].  
\end{equation}

\noindent The following proposition provides commutator estimates in the
norm of the ideal of compact operators which are very similar to those
established in~\cite{AdPS2005} and~\cite{McI1978}.

\begin{pro} 
  \label{infestimate} For any function ~$f: \Rl \rightarrow \Rl$ given
  by~\eqref{ffunc}, there exists an absolute constant $K_0$ such
  that for every $m \ge3$ and for every scalar $0< p\le 1$ the
  following estimates hold
  \begin{enumerate} \item $\ds\|[B,A]\|_{\cC^\infty} \le \pi$,
  \item $\ds\|[f(pB),A]\|_{\cC^\infty} \ge {pK_0\over h(m-\log p)} \log
    {m\over2}$. \end{enumerate} \end{pro}

\begin{proof} 
  The proof of the first claim is based on the norm estimates of the
  Hilbert matrix, see~\cite[the proof of Lemma~3.6]{AdPS2005}.  Hence,
  we need to establish only the second one.  Let us first note, since
  the function $f$ is even, it follows from definition of matrices
  $A$, $B$ that $$ f(pB)A - Af(pB) = \left[\begin{matrix} 0 & f(pD)V -
      Vf(pD)\cr f(pD)V - Vf(pD) & 0\cr \end{matrix} \right], $$ so
  \begin{equation} 
    \label{aaainfest} \|[f(pB),A]\|_{\cC^\infty} =
    \|[f(pD), V]\|_{\cC^\infty}.
  \end{equation} 
  If $S=\{s_{jk}\}_{j,k=1}^m = f(pD)V - Vf(pD)\in\M m$, then $$
  s_{kj}=s_{jk} = {f(pe^{-j}) - f(pe^{-k})\over(e^{-j}+e^{-k})(k-j)}
  \ge 0,\ \ 1\le j, k\le m. $$ If~$1\le j<k\le m$, then, since
  functions $h(t)$ and $e^t$ are monotone, we have
  \begin{align*} 
    s_{jk} & = \biggl({pe^{-j}\over h(j-\log p)} - { pe^{-k} \over
      h(k-\log p)}\biggr) (e^{-j} + e^{-k})^{-1} (k-j)^{-1}\cr & \ge
    {p(e^{-j}-e^{-k})\over h(k-\log p)} (2e^{-j}(k-j))^{-1}\cr &\ge
    {p(1-e^{-1})\over 2h(k-\log p) (k-j)} \ge {p(1-e^{-1})\over
      2h(m-\log p) (k-j)}.\cr 
  \end{align*} 
  Now, using $\sum_{j=1}^{k-1} {1\over j} \ge \log k$, we have $$
  \sum_{j=1}^m s_{jk} \ge \sum_{j=1}^{k-1} s_{jk} \ge
  {p(1-e^{-1})\over 2 h(m-\log p)} \sum_{j=1}^{k-1} {1\over k-j} \ge
  {p(1-e^{-1})\over 2h(m-\log p)} \log k. $$

  Finally, letting $x=(m^{-1/2}, m^{-1/2},\ldots , m^{-1/2})\in \Cx^m$,
  we obtain
  \begin{align*}
    \|S\|_{\cC^\infty} &\, \ge \langle Sx, x \rangle = \frac 1m
    \sum_{j,k = 1}^m s_{jk} \ge \frac 1m \frac {p ( 1 - e^{-1})}{2\,
      h(m - \log p)} \, \sum_{k =1}^m \log k \cr \ge &\, \frac 1m
    \frac{p (1  - e^{-1})} { 2\, h(m - \log p)} \, \sum_{k = [ m /2]}^m
    \log k \cr \ge &\, \frac 1m \frac {p (1 - e^{-1})}{2\, h(m - \log
        p)} \, \frac m 2 \log \frac m 2.
  \end{align*}
  Setting $K_0 = (1-e^{-1})/4$, we have $$ \|[f(pD),V]\|_{\cC^\infty} =
  \|S\|_{\cC^\infty} \ge {pK_0\over h(m-\log p)} \log {m\over2}. $$
  which, together with~\eqref{aaainfest}, completes the proof.
\end{proof}

%%%%%%%%%%%%%%%%%%%%%%%%%%%%%%%%%%%%%%%%%%%%%%%%%%

\noindent 
Together with~\eqref{schurcomm}, Proposition~\ref{infestimate} provides
a lower estimate for the operator norm of Schur multiplier associated
with the function $f$, given by~\eqref{ffunc}, and diagonal matrix $pB$
given by~\eqref{DV} and~\eqref{AB} for every scalar $0<p\le1$ and every
integer $m\ge3$.  Now we extend that lower estimate to a larger class
of ideals.

\begin{pro} 
  \label{Eestimate} Let $E$ be a separable symmetric function space
  with trivial Boyd indices. For every $m\ge3$, let $A_m, B_m\in\M{2m}$
  be given by~\eqref{DV} and~\eqref{AB}, $\Phi_{2m}$, $\Psi_{2m}$ be the
  operators from the Proposition~\ref{phipsi} for the $\varepsilon=1/2$.
  There exists an absolute constant $K_1$ such that for every scalar
  sequence $0<p_m\le1$, and for the sequence of the diagonal matrices
  $W_m = \Phi_{2m}(p_mB_m)\in\M{k_m}$ the following estimate holds $$
  \|\Mf{W_m}\|_{\cC^E\mapsto\cC^E} \ge {K_1\over h(m-\log p_m)} \log
  {m\over2},\ \ m\ge3, $$ where~$f: \Rl \rightarrow \Rl$ is an
  arbitrary function given by~\eqref{ffunc}.
\end{pro}

\begin{proof} Letting $$ X^\infty_m = [p_mB_m, {1\over p_m} A_m ],\ \
  m\ge3, $$ we infer from Proposition~\ref{infestimate} and
  from~\eqref{schurcomm} that for every $m\ge3$ $$
  \|X^\infty_m\|_{\cC^\infty} \le \pi, $$
\begin{align*} \|\Mf{p_mB_m} (X^\infty_m)\|_{\cC^\infty} & =
  \|[f(p_mB_m),{1\over p_m} A_m]\|_{\cC^\infty} \cr & \ge {K_0\over
    h(m-\log p_m)} \log {m\over 2}.\cr \end{align*} It follows from the
definition of Schur multiplication and duality that 
\begin{align*}
  \|\Mf{p_mB_m}\|_{\cC^1\mapsto\cC^1} = &\,
  \|\Mf{p_mB_m}\|_{\cC^\infty\mapsto\cC^\infty} \cr \ge &\, {K_0\over\pi
    h(m-\log p_m)} \log {m\over2},\ \ m\ge3.
\end{align*}
The last estimate implies that there exists a sequence of
$X_m^1\in\M{2m}$ such that $$
{\|\Mf{p_mB_m}(X^1_m)\|_{\cC^1}\over
  \|X^1_m\|_{\cC^1}} \ge {K_0\over 2\pi\,h(m-\log p_m)} \log {m\over2},\ 
\ m\ge3. $$
Suppose now, that~$\alpha_E=0$ and set~$X_m =
\Psi_{2m}(X^\infty_m)$ for every $m\ge3$. It follows from
Proposition~\ref{phipsi} that, for every $m\ge3$, $W_m$ is a finite
diagonal self-adjoint matrix such that
\begin{align*}
  \|\Mf{W_m}\|_{\cC^E\mapsto\cC^E} & \ge {\|\Mf{W_m}(X_m)\|_{\cC^E}\over
    \|X_m\|_{\cC^E}}=
  {\|\Psi_{2m}\{\Mf{p_mB_m}(X^\infty_m)\}\|_{\cC^E}\over
    \|\Psi_{2m}(X^\infty_m)\|_{\cC^E}} \cr & \ge
  {2\,\|\Mf{p_mB_m}(X^\infty_m)\|_{\cC^\infty}\over 3
    \|X^\infty_m\|_{\cC^\infty}} \ge {2K_0\over 3\pi h(m-\log p_m)} \log
  {m\over2}. \cr
\end{align*} 
If we put $K_1 = 2K_0/(3\pi)$, that completes the proof of the case
$\alpha_E=0$.  The only difference in treating the case $\beta_E=1$ is
that we need to use $X^1_m$ instead of $X^\infty_m$ in the above
estimates.\end{proof}

\noindent The following proposition proves that if a function~$f: \Rl
\rightarrow \Rl$ is given by~\eqref{ffunc} and the multipliers~$\Mf
{W_m}$ are not uniformly bounded in~$\cC^E$, then this function is not
commutator bounded in the sense of~\cite{KiShI}.

\begin{pro} 
  \label{tocomm} Let $E$ be a separable symmetric function space.  If
  $f$ is a $C^1$-function and $W_m\in\M{k_m}$ is a sequence of
  diagonal matrices ($m\ge3$) such that \begin{equation}
    \|\Mf{W_m}\|_{\cC^E\mapsto\cC^E} \rightarrow \infty,
    \label{consttending} \end{equation}
  then there exist self-adjoint operators $W$,
  $X$, acting on~$\ell_2$, such that $$
  [W,X]\in\cC^E,\ \ \ [f(W),X]\notin\cC^E. $$
  If, in addition, the norms
  $\|W_m\|_{\cC^\infty}$ are uniformly bounded, then $W(Dom\ X)\subseteq
  Dom\ X$, and if the following series converges $$
  \sum_{m\ge3}
  \|W_m\|_{\cC^E}, $$
  then operator $W$ belongs to $\cC^E$ and $$
  \|W\|_{\cC^E} \le \sum_{m\ge3} \|W_m\|_{\cC^E}. $$
\end{pro}

\def\x#1{X^{\scriptscriptstyle(#1)}_r} 

\begin{proof} It follows from~\eqref{consttending} that there exists a
  subsequence of positive integers $m_r$ ($r\ge1$) and a sequence of
  self-adjoint matrices $\x1\in\M{k'_r}$ such that 
  \begin{equation} \label{aablowestimate} 
    \|\Mf{W'_r}(\x1)\|_{\cC^E}
    \ge 2r^3 \|\x1\|_{\cC^E},\ \ r\ge1,
  \end{equation} 
  where we let, for brevity, $k'_r=k_{m_r}$ and
  $W'_r=W_{m_r}\in\M{k'_r}$.  Let $r\ge1$ be fixed, let
  $\{\lambda_j\}_{j=1}^{k'_r}$ be the sequence of eigenvalues of the
  matrix $W'_r$, and let $\{P_j\}_{j=1}^{k'_r}$ be the collection of
  corresponding one-dimensional spectral projections.
  For~$\lambda\in\Rl$, we set $$ Q_\lambda = \sum_{\ss 1\le j\le
    k'_r\atop\ss\lambda_j=\lambda} P_j. $$ There are only a finite
  number of non-zero projections among
  $\{Q_\lambda\}_{\lambda\in\Rl}$, let us denote them as
  $\{Q_j\}_{j=1}^s$, $1\le s \le k'_r$ and the corresponding sequence
  of eigenvalues as $\{\lambda'_j\}_{j=1}^s$, the scalars $\lambda'_j$
  are mutually distinct.  We consider the self-adjoint matrices $$
  \hat X_r = \sum_{j=1}^s Q_j\x1Q_j,\ \ \text{and}\ \ \x2=\x1-\hat
  X_r. $$ It follows from~\eqref{schurprojections} that (recall
  that~$\psi_f(\lambda, \lambda) = 0$)
  \begin{align*} 
    \Mf{W'_r}(\hat X_r) & = \sum_{1\le j,l\le k'_r}
    \psi_f(\lambda_j,\lambda_l) P_j\hat X_rP_l \cr & = \sum_{t=1}^s
    \sum_{1\le j, l\le k'_r} \psi_f(\lambda_j,\lambda_l)
    P_jQ_t\x1Q_tP_l \cr & = \sum_{t=1}^s \sum_{\ss 1\le j, l\le k'_r
      \atop \ss \lambda_j=\lambda_l=\lambda'_t}
    \psi_f(\lambda_t,\lambda_t) Q_t\x1Q_t = 0, \cr 
  \end{align*}
  and so
  \begin{equation}
    \label{OneTwoEq}
    \Mf{W'_r}(\x2) = \Mf{W'_r}(\x1). 
  \end{equation}
  Now, noting that~$\|\hat X_r\|_{\cC^E} \le \|\x1\|_{\cC^E}$
  (see~\cite[Theorem~III.4.2]{GohbergKrein}) and,
  hence $\|\x2\|_{\cC^E}\le 2\|\x1\|_{\cC^E}$, $\x2$, we infer
  from~\eqref{OneTwoEq} and~\eqref{aablowestimate}
  \begin{equation} 
    \label{abblowestimate} \| \Mf{W'_r}(\x2)\|_{\cC^E}
    \ge r^3\|\x2\|_{\cC^E}.
 \end{equation} 
 We set $$ \x3 = \sum_{1\le j,l\le k'_r} \lambda_{jl} P_j\x2P_l,
 $$ where $$ \lambda_{jl} = \begin{cases} \ds 0, &
   \text{$\lambda_j=\lambda_l$,}\cr \ds {-i\over \lambda_j-\lambda_l},
   & \text{$\lambda_j\ne\lambda_l$.}\cr \end{cases} $$ The matrix
 $\x3$ is self-adjoint and
 \begin{align} \label{aabmakecomm} \x2 & = \sum_{1\le j, l\le k'_r}
   P_j\x2P_l = i \sum_{1\le j, l\le k'_r} \lambda_{jl}
   (\lambda_j-\lambda_l)P_j\x2P_l \cr & = i \sum_{1\le j, l\le k'_r}
   \lambda_{jl} P_j(W'_r\x2-\x2W'_r)P_l \cr & = i \biggl[ W'_r ,
   \sum_{1\le j, l\le k'_r} \lambda_{jl} P_j\x2P_l\biggr] =
   i[W'_r,\x3]. \cr
 \end{align}
 Finally, we let 
 \begin{equation} 
   \label{aabnormirovka} X_r =
   r^{-2}\|\x2\|_{\cC^E}^{-1} \x3.
 \end{equation} 
 For every $r\ge1$ we have constructed so far the finite self-adjoint
 matrices $W'_r$, $X_r$ such that
 \begin{align} \label{aabfirstcomm} \|[W'_r, X_r]\|_{\cC^\infty} & \le
   \|[W'_r, X_r]\|_{\cC^E} \eq{\eqref{aabnormirovka}}
   r^{-2}\|\x2\|_{\cC^E}^{-1} \|[W'_r, \x3]\|_{\cC^E} \cr &
   \eq{\eqref{aabmakecomm}} r^{-2}\|\x2\|_{\cC^E}^{-1} \|\x2\|_{\cC^E}
   = {1\over r^2}\cr
 \end{align}
 and
 \begin{align} \label{aabsecondcomm} \|[f(W'_r), X_r] \|_{\cC^E} &
   \eq{\eqref{aabnormirovka}} r^{-2}\|\x2\|_{\cC^E}^{-1}\|[f(W'_r),
   \x3]\|_{\cC^E} \cr & \eq{\eqref{schurcomm}}
   r^{-2}\|\x2\|_{\cC^E}^{-1} \|\Mf{W'_r}([W'_r,\x3])\|_{\cC^E} \cr &
   \eq{\eqref{aabmakecomm}} r^{-2}\|\x2\|_{\cC^E}^{-1}
   \|\Mf{W'_r}(\x2)\|_{\cC^E} \cr & \geq{\eqref{abblowestimate}}
   r\|\x2\|_{\cC^E}^{-1} \|\x2\|_{\cC^E} \ge r. \cr
\end{align}

\noindent Now, we set $\H=\bigoplus_{r\ge1}\Cx^{k'_r}$,
$X=\bigoplus_{r\ge1} X_r$ and $W=\bigoplus_{r\ge1} W'_r$. Recall, that
by the definition we have 
\begin{align*} \H & = \{\{\xi_r\}_{r\ge1}:\ \ 
  \xi_r\in\Cx^{k'_r},\ \sum_{r\ge1} \|\xi_r\|^2 <\infty\},\cr Dom\ X & =
  \{ \xi=\{\xi_r\}_{r\ge1}\in\H:\ \ X(\xi) =
  \{X_r(\xi_r)\}_{r\ge1}\in\H\}, \cr Dom\ W & =
  \{ \xi=\{\xi_r\}_{r\ge1}\in\H:\ \ W(\xi) =
  \{W_r(\xi_r)\}_{r\ge1}\in\H\}.\cr
\end{align*} 
$W$, $X$ are self-adjoint
operators, acting on the separable Hilbert space $\H$ and $$
\|[W,X]\|_{\cC^E} \le \sum_{r\ge1}
\|[W'_r,X_r]\|_{\cC^E}\leq{\eqref{aabfirstcomm}} \sum_{r\ge1} {1\over r^2}
< \infty, $$
$$
\|[f(W),X]\|_{\cC^E} \ge \max_{r\ge1}
\|[f(W'_r),X_r]\|_{\cC^E} \eq{\eqref{aabsecondcomm}} \infty.  $$
If we
assume that $$
\sum_{m\ge3} \|W_m\|_{\cC^E} < \infty, $$
then $$
\|W\|_{\cC^E} \le \sum_{r\ge1} \|W'_r\|_{\cC^E} \le \sum_{m\ge3}
\|W_m\|_{\cC^E} <\infty. $$
If we assume that
$\sup_{m\ge1}\|W_m\|_{\cC^\infty}\le M<\infty$, then,
by~\eqref{aabfirstcomm}, for every $\xi=[\xi_r]_{r\ge1}\in Dom\ X$,
\begin{align*}
  \biggl(\sum_{r\ge1} \|X_r(W'_r(\xi_r))\|^2\biggr)^{1\over2} & =
  \biggl(\sum_{r\ge1} \|W'_r(X_r(\xi_r)) - [W'_r,X_r](\xi_r)
  \|^2\biggr)^{1\over2} \cr & \le M \biggl(\sum_{r\ge1}
  \|X_r(\xi_r)\|^2\biggr)^{1\over2} \cr & + \sup_{r\ge1}
  \|[W'_r,X_r]\|_{\cC^\infty} \biggl(\sum_{r\ge1}
  \|\xi_r\|^2\biggr)^{1\over2} < \infty.\cr
\end{align*} 
Hence $W(\xi)\in Dom\ X$.  The claim is proved.
\end{proof}

\noindent It follows from Propositions~\ref{Eestimate}
and~\ref{tocomm} that any function~$f: \Rl \rightarrow \Rl$ given
by~\eqref{ffunc} with the function~$h$ satisfying the
condition~$\frac{\log (m/2)}{h(m - \log p_m)} \rightarrow \infty$,
as~$m \rightarrow \infty$ (here~$\{p_m\}_{m\ge 0}$ is some scalar
sequence satisfying~$0< p_m \le 1$) is not commutator bounded in any
separable symmetrically normed ideal with trivial Boyd indices.  In
other words, there exist self-adjoint operators $W$, $X$, acting on a
separable Hilbert space~$\H$, such that $[W, X]\in\cC^E$ but $[f(W),
X]\notin\cC^E$.  We shall now show how further adjustments to the
choice of the function~$h$ and the sequence~$\{p_m\}_{m\ge0}$ can be
made in order to guarantee that the operator~$W$ above belongs
to~$\cC^E$.  First, we need the following auxiliary results.

%%%%%%%%%%%%%%%%%%%%%%%%%%%%%%%%%%%%%%%%%%%%%%%%%%

\begin{pro} 
  \label{fchi} For every $\varepsilon>0$ there exists a function
  $\chi_\varepsilon$ such that 
  \begin{enumerate}
  \item $\chi_\varepsilon\in C^1(\Rl)$, 
  \item $\chi_\varepsilon(t) = 0$, if $t\le 0$, 
  \item $\chi_\varepsilon(t) = 1$, if $t\ge 1$,
  \item $0\le \chi'_\varepsilon \le 1+\varepsilon$.
  \end{enumerate}
\end{pro}

\begin{proof} Let~$\xi_\varepsilon(t)$ be the continuous function such
  that~$\xi_\epsilon (t) = 0$, if~$t\le 0$ or~$t \ge 1$, $\xi_\epsilon
  (t) = 1 + \epsilon$, if~$ \epsilon/(1+\epsilon) \le t \le
  1/(1+\epsilon)$ and linear elsewhere.  It then follows, that the
  function $$ \chi_\varepsilon(t) = \int_{-\infty}^t
  \xi_\varepsilon(\tau)\, d\tau,\ \ t\in\Rl, $$ satisfies the
  assertion.
\end{proof} 

\begin{pro} 
  \label{hconst} Let $s_m$, $q_m$ ($m\ge 0$) be two increasing
  sequences such that 
  \begin{enumerate}
  \item $s_m\rightarrow+\infty$, $s_0=0$,
  \item $q_m\rightarrow+\infty$, $q_0=1$,
  \item $\ds \alpha=\sup_{m\ge1} {\log q_m - \log q_{m-1}\over
      s_m-s_{m-1}} < 1$.\par
  \end{enumerate}
  Then there exists a function $h$ that satisfies the
  conditions~\ref{FL}--\ref{LL} (preceding Proposition~\ref{ffuncPro})
  and such that $h(s_m)=q_m$ for every $m\ge0$.
\end{pro}

\begin{proof} 
  Let $\varepsilon=1/\alpha-1$, and $\chi_\varepsilon$ be the function
  from Proposition~\ref{fchi}. For every $t\ge0$ we define
  \begin{equation} \label{Hdef}  
    H(t) = \sum_{m\ge1} \chi_\varepsilon\Bigl({t-s_{m-1}\over
      s_m-s_{m-1}}\Bigr) (\log q_m - \log q_{m-1}).  
  \end{equation} 
  We have that $x\ge s_{m-1}$ (respectively, $x < s_m$) if and only if
  $$
  {x - s_{m-1}\over s_m-s_{m-1}}\ge 0\ \ \Bigl(\hbox{respectively, }
  {x-s_{m-1}\over s_m-s_{m-1}} < 1\Bigr). $$
  Now, it follows from above
  that for every fixed $t\ge0$ the sum~\eqref{Hdef} is finite, $$
  H(s_0)
  = H(0) = 0 = \log q_0, $$
  and for every $k\ge1$ \begin{align*} H(s_k)
    & = \sum_{m\ge 1} \chi_\varepsilon\Bigl({s_k-s_{m-1}\over
      s_m-s_{m-1}}\Bigr)(\log q_m - \log q_{m-1})\cr & = \sum_{m=1}^k
    (\log q_m - \log q_{m-1}) = \log q_k.\cr \end{align*} We set $h(t)
  := \exp(H(t))$ for every $t\ge0$ and $h(t) = h(-t)$ for every $t<0$,
  then $h(s_m) = q_m$ for every $m\ge0$. Let us check the
  conditions~\ref{FL}--\ref{LL}.

  \toAlpha
  \begin{enumerate}
  \item The function $H$ is a $C^1$-function as a finite sum of
    $C^1$-functions, so $h$ is a $C^1$-function for every $t\ne 0$;
  \item this item holds by the definition of~$h(t)$, and $h(0) =
    \exp(H(0)) =1$;
  \item the function~$H(t)$ is increasing, for every $t\ge0$, as it is
    the sum of increasing functions, so the function $h$ is increasing
    also;
  \item since the sequence $h(s_m) = q_m$ tends to infinity and since
    $h(\cdot)$ is an increasing even function, we have
    $h(\pm\infty)=+\infty$;
  \item for every $t\ge 0$ there exists an integer $k\ge1$ such that
    $s_{k-1}\le t < s_k$, thus it follows from Proposition~\ref{fchi},
    that
    \begin{align*}
      H'(t) & = \sum_{m\ge 1} \chi'_\varepsilon\Bigl({t-s_{m-1}\over
        s_m-s_{m-1}}\Bigr) {\log q_m - \log q_{m-1}\over s_m -
        s_{m-1}},\cr & = \chi'_\varepsilon\Bigl({t-s_{k-1}\over
        s_k-s_{k-1}}\Bigr) {\log q_k - \log q_{k-1}\over s_k-s_{k-1}}
      \cr & \le \alpha(1+\varepsilon) = 1,
    \end{align*}
    and so $$ 0\le H'(t)={h'(t)\over h(t)} \le 1. $$
  \end{enumerate}
  \toRomans
\end{proof}

\noindent Now, we are in a position to prove our main result. 

\begin{thm} 
  \label{main} For every separable symmetric function space $E$ with
  trivial Boyd indices, there exists a $C^1$-function $f_E$,
  self-adjoint operators $W$, $X$, acting on a separable Hilbert
  space~$\H$ such that $$
  W\in\cC^E,\ \ [W,X]\in\cC^E,\ \ W(Dom\ 
  X)\subseteq Dom\ X,\ \ [f_E(W),X]\notin\cC^E.$$
\end{thm}

\begin{proof} Let $m\ge0$, $q_m := (\log(m+e))^{1/2}$, $B_m$ be the
  diagonal matrices, given by~\eqref{DV} and~\eqref{AB}, $\Phi_{2m}$,
  $\Psi_{2m}$ be the operators from Proposition~\ref{phipsi} for
  $\varepsilon=1/2$. Let $\{p_m\}_{m \ge 0}$ be a sequence that
  satisfies the following five conditions
  \begin{enumerate} 
  \item $p_m$ is decreasing to zero;
  \item $0 < p_m\le 1$;
  \item $p_0=1$;
  \item\label{phigrowth} $\ds {1\over p_m} \ge
    m^2\|\Phi_{2m}(B_m)\|_{\cC^E}$, $m\ge 1$;
  \item\label{hprop} $\ds {1\over p_m} \ge {q_m^2\over e q^2_{m-1}}
    \cdot {1\over p_{m-1}}$, $m\ge1$.
  \end{enumerate}
  We construct such a sequence by induction. If the numbers $p_0, p_1,
  \ldots, p_{m-1}$ satisfy the conditions above, then $p_m$ can be
  taken to be any positive number for which $$ {1\over p_m} \ge \max
  \biggl\{1,{1\over p_{m-1}}, m^2\|\Phi_{2m}(B_m)\|_{\cC^E},
  {q_m^2\over e q^2_{m-1}} \cdot {1\over p_{m-1}} \biggr\}. $$\par

  \noindent It follows from~\ref{hprop} above, that $$ {ep_{m-1}\over
    p_m} \ge {q_m^2\over q^2_{m-1}}, $$ and, taking logarithms, $$ 1 +
  \log {1\over p_m} - \log {1\over p_{m-1}} \ge 2(\log q_m - \log
  q_{m-1}). $$ Putting $s_m = m - \log p_m$, we have $$ 0 \le {\log
    q_m - \log q_{m-1}\over s_m - s_{m-1}} \le {1\over 2},\ \ m\ge0.
  $$ Thus, we have verified that the sequences $\{q_m\}_{m\ge0}$ and
  $\{s_m\}_{m\ge0}$ satisfy the conditions of
  Proposition~\ref{hconst}, and so there exists a function $h_E(t)$
  such that $$ h_E(m-\log p_m) = h_E(s_m) = q_m = (\log(e+m))^{1/2}.
  $$

  \noindent If $f_E$ is the function, given by~\eqref{ffunc}, with the
  above choice of $h_E$, $W_m := \Phi_{2m}(p_mB_m)\in\M{k_m}$, where
  $B_m$ given by~\eqref{DV} and~\eqref{AB}, is the finite diagonal
  matrix then it follows from Proposition~\ref{Eestimate} $$
  \|\Mf{W_m}\|_{\cC^E\mapsto\cC^E} \ge {K_1\over h(m-\log p_m)} \log
  {m\over2} = K_1 {\log {(m/2)}\over (\log(m+e))^{1/2}} \rightarrow
  \infty. $$ On the other hand, by our choice of $p_m$
  (see~\eqref{phigrowth} above) we have that $$ \|W_m\|_{\cC^\infty}
  \le \|W_m\|_{\cC^E} = p_m \|\Phi_{2m}(B_m)\|_{\cC^E} \le {1\over
    m^2},\ \ m\ge 3. $$ Now the assertion of Theorem~\ref{main}
  follows from Proposition~\ref{tocomm}.
\end{proof}

\section{Domains of generators of automorphism flows.}

Let $E$ be a separable symmetric function space with trivial Boyd
indices, let~$\cC^E$ be the corresponding symmetrically normed ideal and
let~$X$ be a self-adjoint operator (may be unbounded), acting on a
separable Hilbert space $\H$. We consider a
group~$\alpha=\{\alpha_t\}_{t\in\Rl}$ of automorphisms on $\cC^E$ given
by $$ \alpha(t)T = e^{itX}Te^{-itX},\ \ T\in\cC^E,\ \ t\in\Rl. $$ It
follows from separability of $\cC^E$ that $\alpha$ is a $C_0$-group,
\cite[Corollary~4.3]{PS-RFlow}.  The infinitesimal generator $\delta$ of
$\alpha(t)$ is defined by
\begin{align*} 
  Dom\ \delta & := \biggl\{T\in\cC^E:\ \hbox{ there exists }
  \|\cdot\|_{\cC^E}-\lim_{t\rightarrow 0} {\alpha(t)T - T\over
    t}\biggr\},\cr \delta(T) & :=
  \|\cdot\|_{\cC^E}-\lim_{t\rightarrow0}{\alpha(t)T-T\over t}\ \hbox{
    for every } T\in Dom\ \delta.\cr 
\end{align*}
$\delta$ is a closed densely defined symmetric derivation on $\cC^E$,
i.e.  densely defined closed linear operator such that~$\delta(T^*) =
\delta(T)^*$ and~$\delta(TS) = \delta(T)\, S + T\delta(S)$, for
all~$T,S\in Dom\ \delta$, see e.g.~\cite[Proposition~4.5]{PS-RFlow}.
It is proved in \cite[Proposition~2.2]{AdPS2005}, that $$ Dom\ \delta
= \{T\in\cC^E:\ \ T(Dom\ X)\subseteq Dom\ X,\ [T,X]\in\cC^E\} $$ and
$$ \delta(T) = i[T, X],\ \ T \in Dom\ \delta. $$ Now,
Theorem~\ref{main} yields

\begin{corl} \label{mainderivation} For every separable symmetric
  function space $E$ with trivial Boyd indices, there exists a
  $C^1$-function $f_E$, a self-adjoint operator $W\in\cC^E$ and closed
  densely defined symmetric derivation $\delta$ on~$\cC^E$, such that
  $W\in Dom\ \delta$ but $$ f_E(W)\notin Dom\ \delta.$$
\end{corl}

} %%% END OF THE GROUPING LOCALIZING MY PERSONAL MACROS. DO NOT TOUCH.

% {\small
%   \let\Large=\normalsize
%   \bibliography{references1}}

%%% BIBLIOGRAPHY

\providecommand{\bysame}{\leavevmode\hbox to3em{\hrulefill}\thinspace}

\end{document}